\documentclass{amsart}
\usepackage{graphicx}
\usepackage{graphics}
\usepackage{psfrag}
\usepackage[active]{srcltx}
\usepackage{amssymb}
\usepackage{amscd}

\newtheorem{theorem}{Theorem}[section]
\newtheorem{cor}[theorem]{Corollary}
\newtheorem{lm}[theorem]{Lemma}

\theoremstyle{definition}
\newtheorem{df}[theorem]{Definition}

\newtheorem{prb}[theorem]{Problem}
\newtheorem{con}[theorem]{Convention}

\theoremstyle{remark}

\numberwithin{equation}{section}

\newif\ifShowLabels
\ShowLabelstrue
\newdimen\theight
\newcommand\TeXref[1]{%
     \leavevmode\vadjust{\setbox0=\hbox{{\tt
               \quad\quad #1}}%
                    \theight=\ht0
                         \advance\theight by \dp0
                              \advance\theight by \lineskip
                                   \kern -\theight \vbox to
                                             \theight{\rightline{\rlap{\box0}}%
                                                   \vss}%
                                                         }}%
\newcommand{\labelp}[1]{\label{#1}%
    \ifShowLabels \TeXref{{#1}} \fi}
\ShowLabelsfalse

\newcommand\bt{\beta} 
\newcommand\gm{\gamma}  
\newcommand\g{\gamma}  %
\newcommand\kp{\kappa}  
\newcommand\lph{\alpha}

\newcommand\vph{\varphi}

\newcommand\wi{\xy}
\newcommand\wj{\yz}
\newcommand\wk{\xz}

\newcommand\ww{w}
\newcommand\wx{x}
\newcommand\wy{y}
\newcommand\wz{z}
\newcommand\gxhi{G_\wx/H_\wi}
\newcommand\gyki{G_\wy/K_\wi}

\newcommand\h[1]{H_{#1}}
\newcommand\hs[2]{H_{#1,#2}}
\newcommand\ks[2]{K_{#1,#2}}
\newcommand\e[1]{e_{#1}}
\newcommand\vphi[1]{\vph_{#1}}
\newcommand\pair[2]{(#1,#2)}
\newcommand\kai[1]{\kappa_{#1}}
\newcommand\G[1]{G_{#1}}
\newcommand\gsq[2]{\G{#1}\times \G {#2}}

\newcommand\mo{^{-1}}
\newcommand\seq{\subseteq}
\newcommand\ssm{^{\scriptscriptstyle\smile}}
\newcommand\scir{\raise2pt\hbox{$\,\scriptscriptstyle\circ\,$}}

\newcommand\tbigcup{\textstyle \bigcup}

\newcommand\tsum{\textstyle\sum}
\newcommand\idd[1]{id_{#1}}
\newcommand\id[1]{id_{#1}}
\newcommand\co{\textnormal{,}\ }     
\newcommand\f[1]{{\mathfrak {#1}}}

\newcommand\smbcomma{\,,}
\newcommand{\mc}[1]{\mathcal{#1}}
\renewcommand\k[1]{K_{#1}}
\renewcommand\a{\alpha}
\renewcommand\r[2]{R_{{#1},{#2}}}
\newcommand\per{\textnormal{\myspace.\ }}
\newcommand{\myspace}{{\hspace*{.5pt}}}
\newcommand\qeddef{\qed}
\newcommand{\comma}{\textnormal{,}\ }
\newcommand\xy{{xy}}
\newcommand\cs[2]{{#1}_{#2}}
\newcommand\xx{{xx}}
\newcommand\ex[1]{e_{#1}}
\newcommand\po{\textnormal{.}\ }     
\newcommand\yx{{yx}}
\renewcommand\b{\beta}
\newcommand\rp{\!\mid\!}
\newcommand\wwz{{wz}}
\newcommand\varnot{\varnothing}
\newcommand\yz{{yz}}
\newcommand\xz{{xz}}
\newcommand\hvphs[1]{\hat{\vph}_{#1}}

\newcommand\cra[2]{\mathfrak{#1}\myshortspace[\mathcal{#2}\,]}
\newcommand{\myshortspace}{{\hspace*{.1pt}}}

\newcommand{\refL}[1]{Lemma~\ref{L:#1}}
\newcommand{\refD}[1]{Definition~\ref{D:#1}}
\newcommand{\refC}[1]{Corollary~\ref{C:#1}}
\newcommand{\refT}[1]{Theorem~\ref{T:#1}}

\newcommand\cc[3]{C_{#1#2#3}}
\newcommand\trip[3]{(#1,#2,#3)}
\newcommand\ez[1]{\mc E_{#1}}
\newcommand\hh{\h\xy\scir\h\xz}

\newcommand\ident{1\mynegspace\textnormal{\rq}}
\newcommand{\mynegspace}{{\hspace*{-.5pt}}}

\newcommand\refEq[1]{(\ref{Eq:#1})}
\newcommand\refCo[1]{Convention~\ref{Co:#1}}

\newcommand\vp{\varphi}

\newcommand\vth{\vartheta}

\newcommand\sk[1]{\textsf{#1}}
\newcommand\xox{\wx;1;\wx}
\newcommand\famxi{^{(\xi)}}
\newcommand\comment[1]{\,}
\newcommand\tprod{\textstyle\prod}

\begin{document}

\title[The variety of coset relation algebras]{The variety of coset relation algebras}

\author{Steven Givant and Hajnal Andr\'eka}%
\address{Steven Givant\\Mills College\\5000 MacArthur
Boulevard, Oakland, CA 94613}\email{givant@mills.edu}
\address{Hajnal Andr\'eka\\Alfr\'ed R\'enyi Institute of Mathematics\\
Hungarian Academy of Sciences\\Re\'altanoda utca 13-15\\ Budapest\\ 1053 Hungary}\email{andreka.hajnal@renyi.mta.hu}
\thanks{This research was
partially supported  by Mills College and  the Hungarian National
Foundation for Scientific Research, Grants T30314 and T35192.}

\font\sans=cmss10
\newcommand\rraa{{\text{\sans {RRA}}}}
\newcommand\craa{{\text{\sans {CRA}}}}
\newcommand\graa{{\text{\sans {GRA}}}}
\newcommand\raa{{\text{\sans {RA}}}}

\begin{abstract} Givant\,\cite{giv1}  generalized the notion
of an atomic pair-dense relation algebra from Maddux\,\cite{ma91} by
defining the notion of a \textit{measurable relation algebra}, that
is to say, a relation algebra in which the identity element is a sum
of atoms that can be measured in the sense that the ``size" of each
such atom can be defined in an intuitive and reasonable way (within
the framework of the first-order theory of relation algebras). In
Andr\'eka-Givant~\cite{ag}, a large class of examples of such
algebras is constructed from systems of groups, coordinated systems
of isomorphisms between quotients of the groups, and systems of
cosets that are used to ``shift" the operation of relative
multiplication.  In Givant-Andr\'eka~\cite{ga}, it is shown that the
class of these \textit{full coset relation algebras}  is adequate to
the task of  describing all measurable relation algebras in the
sense that every atomic and complete measurable relation algebra is
isomorphic to a full coset relation algebra.

Call an algebra $\f A$ a \emph{coset relation algebra} if $\f A$ is
embeddable into some full coset relation algebra.  In the present
paper, it is shown that the class of coset relation algebras is
equationally axiomatizable (that is to say, it is a variety), but
that no finite set of sentences suffices to axiomatize the class
(that is to say, the class is not finitely axiomatizable).
\end{abstract}

\maketitle

\section{Introduction}\labelp{S:sec1}
In \cite{giv1},  a subidentity  element $x$---that is to say, an
element below the identity element---of a relation algebra is
defined to be \textit{measurable} if it is an atom and if the square
$x;1;x$ is a sum of  functional elements,
that is to say, a set of abstract elements $f$ satisfying the
functional inequality $f\ssm;f\le \ident$. (A functional element is
an abstract version of a function in that in a concrete algebra of
binary relations an element is functional if and only if it is a
function set theoretically, i.e., $(u,v)\in f$ and $(u,w)\in f$
imply $v=w$.)
The number of non-zero functional elements below the square $x;1;x$
gives the \textit{measure}, or the \textit{size}, of the atom $x$. A
relation algebra is said to be \textit{measurable} if the identity
element is the sum of measurable atoms, and \emph{finitely
measurable} if each of the measurable atoms has finite measure.

The group relation algebras  constructed in \cite{giv1}  are
examples of measurable relation algebras.  Interestingly, the class
\graa\ of algebras embeddable into the full group relation algebras
coincides with the variety \rraa\  of all representable relation
algebras \cite[section 5]{giv1}, in symbols
\[ \graa =\rraa \per\]
It turns out that full group relation algebras are not the only
examples of measurable relation algebras. In \cite{ag},  a more
general class of measurable relation algebras is constructed.  The
algebras are obtained from group relation algebras by ``shifting"
the relational composition operation by means of coset
multiplication, using an auxiliary system of cosets. For that
reason, they are called \textit{full coset relation algebras}, and
they are not too much of a distortion to representable algebras.
They are a genuine generalization to group relation algebras,
because among them are algebras that are not representable
\cite[Thm.5.2]{ag}. However, this class is adequate to the task of
describing all atomic, complete measurable relation algebras in the
sense that a relation algebra is atomic, complete and measurable if
and only if it is isomorphic to a full coset relation algebra
\cite[Thm.7.2]{ga}. 

In the present paper, we show that the class \craa\  of algebras
embeddable into full coset relation algebras is a variety.  It
is a generalization of the class \rraa\ of representable relation
algebras. Given the relationship between \graa\ and \rraa, it is
natural to ask whether \craa\ coincides with the class \raa\ of all
relation algebras.  We prove that this is not the case, and in fact
\craa\ is not finitely axiomatizable as \raa\ is. Thus
\[ \graa =\rraa \subset \craa \subset \raa\per\]
Thus $\craa$ shares the properties of $\rraa$ of being a variety and
of being not finitely axiomatizable.  

An extended abstract describing the above results and their
interconnections was published by the authors in \cite{ga02}. The
reader may find the expository and motivational material of
\cite{ga02} helpful in connection with the present paper. Readers
who wish to learn more about the subject of relation algebras and
their connection to logic are recommended to look at one or more of
the books Hirsch-Hodkinson\,\cite{hh02}, Maddux\,\cite{ma06},
Givant\,\cite{giv18, giv18b}, or Tarski-Givant\, \cite{TG}.

\section{Group and coset relation algebras}\labelp{S:sec2}

Here is a summary of  the essential notions from \cite{giv1,ag,ga02}
that will be needed in this paper. Fix a system
\[G=\langle
\G x:x\in I\,\rangle\] of  groups
that are pairwise disjoint, and an associated system
\[\varphi=\langle\vph_{xy}:\pair x y\in \mc E\,\rangle\] of
isomorphisms between quotient groups.  Specifically, we require that
$\mc E$ be an equivalence relation on the index set $I$, and for
each pair $\pair x y$ in $\mc E$, the function $\vphi {xy}$  be an
isomorphism from a quotient group of $\G x$ to a quotient group of
$\G y$. Call
\[\mc F=\pair G \varphi\]  a \textit{group pair}.  The set    $I$ is
the \textit{group index set},  and the equivalence relation $\mc
E$ is the
 (\textit{quotient}) \textit{isomorphism index set} of $\mc F$.
The normal subgroups of $\G x$ and $\G y$ from which the quotient
groups are constructed are uniquely determined by
 $\vphi\wi$, and will be denoted  by $\h\wi$ and $\k\wi$
respectively, so that $\vphi\wi$ maps $\gxhi$ isomorphically onto
$\gyki$.

Let $\kai\wi$ denote the cardinality of the quotient group  $\gxhi$.
For a fixed enumeration $\langle \hs\wi \g:\g<\kai\wi\rangle$
(without repetitions) of the cosets of $\h\wi$ in $\G x$, the
isomorphism $\vphi\wi$ induces a \textit{corresponding}, or
\textit{associated}, coset system
 of $\k\wi$ in $\G y$,  determined by the rule
\[\ks\wi \g=\vphi\wi(\hs\wi \g)\]
for each $\gm<\kai\wi$. In what follows, it is always assumed that
the given coset systems for $\h\wi$ in $\G x$ and for $\k\wi$ in $\G
y$ are associated in this manner. Further, we will always assume
that the coset $\h\wi$ is the first one in the enumeration: $\hs\wi
0=\h\wi$.
In the following, $\scir$ will denote the group operations of the
groups in question, we hope context will always tell which group we
have in mind.
\begin{df} \labelp{D:compro}  For each pair $\pair
x y$  in $\mc E$ and each $\a<\kai\wi$, define a  binary relation
$\r\wi \a $ by\[ \r{\wi}{\lph}= \tbigcup_{\gm < \kp_\wi}
H_{\wi,\gm}\times \vphi\wi[H_{\wi,\gm}\scir H_{\wi,\lph}]=
\tbigcup_{\gm < \kp_\wi} H_{\wi,\gm}\times (K_{\wi,\gm}\scir
K_{\wi,\lph})\per\] \qeddef\end{df}

The set $A$ of all possible unions of sets of such relations is a
complete Boolean set algebra, but it may not contain the identity
relation, nor need it be closed under the operations of relational
converse and composition.  The following theorems from \cite{giv1}
characterize when we do obtain such closure, so that $A$ is the
universe of a set relation algebra.

\begin{lm}[Partition Lemma] \labelp{L:i-vi}  The
relations $\r\wi \a $\comma for $\a<\kai\xy$\comma are non-empty
and partition the set $\gsq x y$\per
\end{lm}

\begin{theorem}[Boolean Reduct Theorem]\labelp{T:disj} The set $A$ is
the universe of a complete\comma atomic Boolean algebra of sets\per
The  atoms are the relations $\r{\wi}{\lph}$\comma and the elements
in $A$ are the unions of the various  sets of atoms\per
\end{theorem}

In the following, $\ex\wx$ denotes the identity element of the group
$G_x$, and $\id U=\{(u,u) : u\in U\}$ is the identity relation on
the set $U$. Also, we often denote the domain of the group $G_x$
also by $G_x$.

\begin{theorem}[Identity Theorem]\labelp{T:identthm1} For each element $x$ in
$I$\co the following conditions are equivalent\per
\begin{enumerate}
\item[(i)]
The identity relation $\idd {\G x}$ on $\G\wx$ is in $A$\per
\item[(ii)] $\r\xx 0=\idd {\G x}$\per
\item[(iii)]$\vphi\xx$ is the identity
automorphism of $\G\wx/\{\ex \wx\}$\po
\end{enumerate}
Consequently\co the set $A$ contains the identity relation $\id U$
on the base set $U$ if and only if \textnormal{(iii)}  holds for
each $\wx$ in $I$\po
\end{theorem}

\begin{con}\labelp{Co:id11}
Suppose that the identity relation is in $A$. Then $\h {xx}=\{ \e
x\}$ by (iii) of the Identity Theorem. Consequently, the cosets of
$\h {xx}$ are the singletons $\{ g\}$ for $g\in\G x$. We will write
simply $\r {xx} g$ in place of $\r {xx}\g$ for $\g=\{ g\}$. Thus,
for example, $\{\r{xx}g : g\in\h{xy}\}$  means  $\{\r{xx}\g : \g=\{
g\}\mbox{\ for some\ }g\in\h{xy}\}$. Note that $\{\r {xx} g : g\in\G
x\}$ is the same as $\{\r {xx}\g : \g<\kai {xx}\}$, and $\kai
{xx}=|G_x|$.
\end{con}
\renewcommand\wi{\xy}
\renewcommand\wj{\yx}
In the following, $R\mo=\{(v,u) : (u,v)\in R\}$ denotes the inverse
of the binary relation $R$. We also denote by $a\mo$ the inverse of
an element $a$ in a group.

\begin{theorem}[Converse Theorem]\labelp{T:convthm1}  For each
pair $\pair x y$  in $\mc E$\co the following conditions are
equivalent\per
\begin{enumerate}
\item[(i)] There are an $\a<\kai \xy$ and a $\b<\kai\yx$ such that $\r\xy\lph\mo=\r\yx\bt$\per
\item[(ii)] For every $\a<\kai \xy$ there is a $\b<\kai\yx$ such that $\r\xy\lph\mo=\r\yx\bt$\per
\item[(iii)]$\vphi\xy\mo=\vphi\yx$\,\po
\end{enumerate}
Moreover\comma if one of these conditions holds\comma then we may
assume that $\kai \yx=\kai\xy$\comma and the index $\b$ in
\textnormal{(i)} and \textnormal{(ii)} is uniquely determined by the
equation $ \hs\xy\lph\mo=\hs\xy\bt$\per The set $A$ is closed under
converse if and only if \textnormal{(iii)} holds for all   $\pair x
y$ in $\mc E$\po
\end{theorem}

\begin{con}
\labelp{Co:con11} Suppose $A$ is closed under converse. If a pair
$\pair x y$ is in $\mc E$, then $ \h\yx=\k\xy$, and therefore any
coset system for $\h\yx$ is also a coset system for $\k\xy$\per
Since the enumeration $\langle \hs \yx \g:\g<\kai \yx\rangle$ of
the cosets of $\h \yx$ can be freely chosen,
 we can and  always shall choose it so that
$\kai \yx=\kai \xy$ and
 $\hs \yx\g=\ks \xy\g$ for $\g<\kai \xy$\per  It then follows from  the
Converse Theorem that $\ks \yx \g = \hs \xy \g$ for $\g<\kai
\xy$\per
\end{con}
\renewcommand\wi{\xy}
\renewcommand\wj{\wwz}
In the following, $R\rp S=\{(u,w) : (u,v)\in R\mbox{ and }(v,w)\in
S\mbox{ for some }v\}$ denotes the relational composition of the
binary relations $R$ and $S$.
\begin{lm}\labelp{L:emptycomp}  If $\pair \wx\wy$ and $ \pair w z$ are in $\mc E$\co and if $y\ne
w$\co then
\[\r\wi \a \rp \r\wj \b=\varnot
\]
for all $\lph<\kai\wi$ and $\bt <\kai\wj$\po
\end{lm}

\renewcommand\wi{\xy}
\renewcommand\wj{\yz}
\renewcommand\wk{\xz}

The most important case regarding the composition of two atomic
relations is when $y=w$.

\begin{theorem}[Composition Theorem]\labelp{T:compthm}  For
all pairs $\pair x y$ and $\pair y z$  in $\mc E$\co the following
conditions are equivalent\per
\begin{enumerate}
\item[(i)] The relation
$\r\xy 0 \rp\r\yz 0$ is in $A$\per
\item[(ii)] For each
$\lph<\kai\xy$ and each $\bt<\kai\yz$\comma the relation
$\r\xy\lph\rp\r\yz\bt$ is in $A$\per
\item[(iii)] For each
$\lph<\kai\xy$ and each $\bt<\kai\yz$\comma
\[\r \xy \a \rp \r \yz\b=\tbigcup\{\r \xz \g:
\hs \xz \g \seq \vphi \xy\mo[ \ks\xy \a\scir\hs \yz \b]\}\per\]
\item[(iv)]$ \h\xz\seq\vphi\xy\mo[\k\xy\scir\h\yz]$ and
$\hvphs\xy\rp\hvphs\yz=\hvphs\xz$\comma where
$\hvphs\xy$ and $\hvphs\xz$ are the mappings induced by
$\vphi\xy$ and $\vphi\xz$ on the quotient of $\G\wx$ modulo the
normal subgroup $\vphi\xy\mo[\k\xy\scir\h\yz]$\comma while
$\hvphs\yz$ is the isomorphism induced by $\vphi\yz$ on the
quotient of $\G\wy$ modulo the normal subgroup
$\k\xy\scir\h\yz$\per \end{enumerate} Consequently\co the set $A$
is closed under relational composition if and only if
\textnormal{(iv)} holds for all pairs  $\pair x y$ and $\pair y z$
in $\mc E$\po
\end{theorem}

The next theorem clarifies the characters of the mappings induced
by the quotient isomorphism.
\begin{theorem}[Image Theorem]\labelp{T:domainofmap}  If the set $A$ is
closed under converse and composition\co then
\begin{gather*}
\vphi\xy [\h\xy\scir\h\xz]=\k\xy\scir\h\yz\comma\qquad \vphi\yz
[\k\xy\scir\h\yz]=\k\xz\scir\k\yz\comma\\
\vphi\xz[\h\xy\scir\h\xz]=\k\xz\scir\k\yz
\end{gather*}
for all $\pair x y$ and $\pair y z$ in $\mc E$\po
\end{theorem}

Full group relation algebras by themselves are not  sufficient to
represent all atomic, measurable relation algebras, because the
operation of relative multiplication need not coincide with that of
relational composition in the most natural candidate for a
representable copy of a measurable relation algebra
\cite[Thm.5.2]{ag}. The operation in an arbitrary measurable
relation algebra may be a kind of ``shifted" relational composition.
It is therefore necessary to add one more ingredient to a group pair
$\mc F=\pair G\vph$, namely a system of cosets
\[\langle \cc x y z:\trip x y z\in\ez 3\rangle\comma
\]
where $\ez 3$ is the set of all triples $\trip xyz$ such that the
pairs $\pair xy$ and $\pair yz$ are in $\mc E$, and for each such
triple, the set $\cc x y z$ is a coset of the normal subgroup $\hh$
in $\G\wx$.  Call the resulting triple \[\mc F=\trip G \vph C\] a
\textit{group triple}.

Define a new binary multiplication operation $\,\otimes\,$ on the
pairs of atomic relations in the  Boolean algebra $A$ of
\refT{disj}  as follows.

\begin{df}\labelp{D:smult}  For pairs
$\pair \wx\wy$ and $\pair \wy \wz$ in $\mc E$\comma put
\[\r \xy \a \otimes \r \yz\b=\tbigcup\{\r \xz \g:
\hs \xz \g \seq \vphi \xy\mo[ \ks\xy \a\scir\hs \yz \b]\scir\cc x
y z\}\] for all $\lph<\kai \xy$ and all $\bt<\kai\yz$\comma and
for  all other  pairs $\pair \wx\wy$ and $\pair \ww \wz$ in $\mc
E$ with $\wy\neq \ww$, put
\[\r\xy\lph\otimes\r\wwz\bt=\varnothing
\]
for all $\lph<\kai \xy$ and $\bt<\kai\wwz$\per
 Extend $\,\otimes\,$ to all of $A$ by requiring it to distribute over
 arbitrary unions.
 \qed
\end{df}

Comparing  the formula defining $\r\xy\lph\otimes\r\yz\bt$ in
\refD{smult}  with the value of the relational composition
$\r\xy\lph\rp\r\yz\bt$ given in
 Composition \refT{compthm}(iii), it is clear that they are very similar in form. In the
 first case, however,
 the coset
$\vphi\xy\mo[\ks\xy\lph\scir\hs\yz\bt] $ of the composite group
$\h\xy\scir\h\xz$ has been shifted, through coset multiplication by
$\cc  x y z$\comma to another  coset of $\h\xy\scir\h\xz$, so that
in general the value of the $\,\otimes\,$-product and the value of
relational composition on a given pair of atomic relations will be
different, except in certain cases, for example, the case in which
the value is the empty set.

\begin{lm}\labelp{L:cra1} $\r\wi\lph\otimes\r\wi\lph\mo=\r\wi\lph\rp\r\wi\lph\mo=\tbigcup\{\r\xx g:g\in\h\xy\}$\per
\end{lm}\begin{proof} The relation $\r\wi\lph\mo$ is equal to $\r\yx\beta$ for   $\beta$ such that
\begin{equation*}\tag{1}\labelp{Eq:cra1.1}
  \hs\xy\beta=\hs\xy\lph\mo\comma
\end{equation*}
  by   Converse \refT{convthm1}. Note in passing that \refEq{cra1.1} and the isomorphism properties of $\vphi\xy$ imply that
  \[\ks\xy\beta=\ks\xy\lph\mo\comma\] and hence that
  \begin{equation*}\tag{2}\labelp{Eq:cra1.02}
    \hs\yx\beta=\hs\yx\lph\mo\comma
\end{equation*} by   \refCo{con11}.

Lemma 6.5 in \cite{ag} implies that the first equality in
\refEq{cra1.1} holds with $\r\yx\beta$ in place of $\r\wi\lph\mo$ if
and only if $\cc xyx=\h\xy\scir\h\xx=\h\xy$. This last equality does
hold, by the coset conditions listed in Theorem 7.6(v) of \cite{ag},
so the first equality of the lemma holds.

As regards the second equality of the lemma, we have
\begin{equation*}\tag{3}\labelp{Eq:cra1.2}
  \r\xy\lph\rp\r\yx\beta=\tbigcup\{\r\xx g: g\in\vphi\xy\mo[\ks\xy\lph\scir\hs\yx\beta]\}\comma
\end{equation*}
 by the Composition \refT{compthm}. Now $\ks\xy\lph=\hs\yx\lph$, by \refCo{con11}, so
 \begin{equation*}\tag{4}\labelp{Eq:cra1.4}
\ks\xy\lph\scir\hs\yx\beta=\hs\yx\lph\scir\hs\yx\beta=\hs\yx\lph\scir\hs\yx\lph\mo=\h\yx\comma
\end{equation*}
 by \refEq{cra1.02} and the group inverse property.  Consequently,
  \begin{equation*}\tag{5}\labelp{Eq:cra1.5}
\vphi\xy\mo[\ks\xy\lph\scir\hs\yx\beta]=\vphi\xy\mo(\h\yx)=\vphi\xy\mo(\k\xy)=\h\xy\comma
\end{equation*}
by \refEq{cra1.4}, \refCo{con11}, and the definition of $\vphi\xy$.  Replace the left side of \refEq{cra1.5} in \refEq{cra1.1} by the right side of \refEq{cra1.5} to arrive at the second equality of the lemma.
\end{proof}

\begin{lm}\labelp{L:cra2} $(\G x\times\G y)\otimes(\G y\times\G z)=(\G x\times\G y)\rp (\G y\times\G z)=\G x\times\G z$\per
\end{lm}\begin{proof} The second equality is obviously true.  To derive the first equality, it is helpful to derive the second equality in a more roundabout way. Use Partition \refL{i-vi}, the distributivity of relational composition over unions, and Composition \refT{compthm}, to obtain
\begin{align*}(\G x\times\G y)\rp(\G y \times\G z)&=(\tbigcup\{\r\xy\lph:\lph<\kai\xy\})\rp (\tbigcup\{\r\yz\bt:\bt<\kai\yz\})\\
&= \tbigcup\{\r\xy\lph\rp \r\yz\bt:\lph<\kai\xy \text{ and }\bt<\kai\yz\}\\
&= \tbigcup\{\r\xz\gm:\hs\xz\gm\seq\vphi\xy\mo[\ks\xy\lph\scir\hs\yz\bt], \lph<\kai\xy, \bt<\kai\yz\}\per
 \end{align*}   As $\lph$ and $\bt$ vary over their index sets, the cosets $\ks\xy\lph\scir\hs\yz\bt$ of $\k\xy\scir\h\yz$ in $\G y$
 vary over all of the
 cosets of $\k\xy\scir\h\yz$, the union of which is just $\G y$. Continue with the preceding string of equalities  to arrive at
\begin{align*}\tag{6}\labelp{Eq:cra2.7}(\G x\times\G y)\rp(\G y \times\G z)
&= \tbigcup\{\r\xz\gm:\hs\xz\gm\seq\vphi\xy\mo[\G y]\}\\&=\tbigcup\{\r\xz\gm:\hs\xz\gm\seq\G x\}\\&=\tbigcup\{\r\xz\gm:\gm<\kai\xz\}\\
&=\G x\times\G z\per
 \end{align*}

 The computation with $\,\otimes\,$ in place of $\,\rp\,$ is nearly the same, but the composition with $\cc xyz$ must be adjoined on the right to each of  the terms
 \[\vphi\xy\mo[\ks\xy\lph\scir\hs\yz\bt], \qquad\vphi\xy\mo[\G y], \qquad \G x\comma\] that is to say, these three terms must be replaced by
  \[\vphi\xy\mo[\ks\xy\lph\scir\hs\yz\bt]\scir\cc xyz, \qquad\vphi\xy\mo[\G y]\scir\cc xyz, \qquad \G x\scir\cc xyz\] respectively. Note that $\G x=\G x\scir\cc xyz$, so we arrive at the same final equality.
  Combine these observations to obtain the first equality of the lemma.
\end{proof}

\section{The variety generated by the class of full coset relation
algebras}

Call an algebra $\f A$ a \emph{coset relation algebra} if $\f A$ is
embeddable into a full coset relation algebra, and let \craa\  be
the class of all coset relation algebras. The class \craa\ is an
analogue of \rraa . A rather surprising consequence of the
Representation Theorem for measurable relation algebras
\cite[Theorem 7.4]{ga} is that the class \craa \ is equationally
axiomatizable, or a \textit{variety}, as such classes are usually
called. The proof of this theorem is analogous to the proof of
Tarski's theorem in \cite{t55} that the class of representable
relation algebras forms a variety.

\begin{theorem}\labelp{T:var} The class of coset relation algebras is a variety\per
\end{theorem}
\begin{proof} Let $\sk K$ be the class of all atomic,
measurable relation algebras, and  denote by $\sk S(\sk K)$ the
class of algebras that are embeddable into some algebra in $\sk K$.
The first step in proving the theorem is to show that the class $\sk
K $ is first-order axiomatizable.  In other words, there is a set
$\Gamma$ of first-order sentences such that an algebra $\f A$ is in
$\sk K$ just in case $\f A$ is a model of $\Gamma$, that is to say,
just in case all the sentences of $\Gamma$ are true of $\f A$, where
everything is taken in the signature of relation algebras.

First, put the relation algebraic axioms into $\Gamma$. Next,
observe that the property of being an atom is expressible in
first-order logic: an atom is a minimal non-zero element.
Consequently, the property of being an atomic algebra is expressible
by a first-order sentence $\vph$  saying that below every non-zero
element there is an atom. Put $\vph$ into $\Gamma$.  The property of
being a measurable atom is also first-order expressible as follows:
an element $x$ is a measurable atom just in case $\wx$ is a
subidentity atom (an atom below $\ident$), and  every non-zero
element below $\xox$ is above some non-zero  functional element (an
element $f$ satisfying the functional inequality
$f\ssm;f\le\ident$).  The first-order sentence $\psi$ stating that
below every non-zero subidentity element there is a measurable atom
expresses the property of an algebra being measurable.  Put $\psi$
into $\Gamma$. Clearly, $\Gamma$ is a set of axioms for $\sk K$, in
symbols,
\begin{equation*}\tag{1}\labelp{Eq:var1}
  \sk{{Mo}}(\Gamma)=\sk K\comma
\end{equation*}
where $\sk{{Mo}}(\Gamma)$ is the class of all models of $\Gamma$.
Let $\Theta$ be the set of universal sentences true in $\sk K$. A
well-known theorem of Tarski \cite{t54.2} says that, for any
first-order axiomatizable class $\sk L$ of algebras, the class $\sk
S(\sk L)$ of  algebras embeddable into algebras of $\sk L$ is
axiomatizable by a set of (first-order) universal sentences. In
particular,
\begin{equation*}\tag{2}\labelp{Eq:var2}
  \sk{{Mo}}(\Theta)=\sk S(\sk K)\per
\end{equation*}

The next step is to prove that
\begin{equation*}\tag{3}\labelp{Eq:var3}
 \craa\ =\sk S(\sk K)\per
\end{equation*}
 Every full coset relation algebra is in the class $\sk K$ (this is proved in \cite{ag}).  Consequently, every
 coset relation algebra is in $\sk S(\sk K) $, because this class is closed under subalgebras and isomorphic images.
 This establishes the inclusion from left to right in \refEq{var3}.
 To establish the reverse inclusion, use the representation theorem for measurable relation algebras \cite[Theorem 7.4]{ga}.
 This theorem says that every algebra in $\sk K$ is
 embeddable into a full coset relation algebra, and consequently belongs to the class \craa\ of all coset relation algebras.
 It follows that every algebra in $\sk S(\sk K)$ is embeddable into an algebra in the class \craa\ and therefore belongs
 to this class, because the class is closed under subalgebras and isomorphic images. This completes the proof of \refEq{var3}.

The remarks after the proof of Theorem 6.1 in \cite[p.51]{ag} imply
that the direct product of a system of full coset relation algebras
is isomorphic to a full coset relation algebra. It follows that the
direct product of a system of coset relation algebras is embeddable
into a full coset relation algebra. Thus, $\sk S(\sk K)$ is closed
under direct products, and consequently under subdirect products.

  Consider again the set $\Theta$ of universal sentences that axiomatizes $\sk
S(\sk K) $.    It is a well-known theorem in the theory of relation algebras (due to  Tarski---see Theorem 9.5 in \cite{giv18}) that for
every universal sentence $\theta$ in the  language of relation algebras, there
is an (effectively constructible)  equation $\varepsilon_\theta$
in the language of relation algebras such that $\theta$ and
$\varepsilon_\theta$ are equivalent in all simple relation algebras, that  is to say,
$\theta$ is valid in a simple relation algebra $\f A$ just in case
$\varepsilon_\theta$ is valid in $\f A$.  Let $\Delta$
be the set of equations corresponding to  universal sentences in
$\Theta$,
\[\{\varepsilon_\theta:\theta\in\Theta\}\comma
\] together with the axioms of the theory of relation algebras.
\begin{equation*}\tag{4}\labelp{Eq:var4}
  \sk{{Mo}}(\Delta)=\sk S(\sk K)\per
\end{equation*}

To prove \refEq{var4}, consider any model $\f A$  of $\Delta$.
Certainly, $\f A$ is a relation algebra, because the relation
algebraic axioms are all in $\Delta$.  Every relation algebra is
isomorphic to a subdirect product of simple relation algebras (see
Theorem 12.10 in \cite{giv18}).  Let $\f B$ be a simple subdirect
factor of $\f A$.  Since $\f B$ is a homomorphic image of $\f A$,
every equation true of $\f A$ is true of $\f B$.  (Recall that
equations are preserved under the passage to homomorphic images.) It
follows that each equation in $\Delta$ is valid in $\f B$.  Now $\f
B$ is simple, by assumption, so each sentence in $\Theta$ is valid
in $\f B$. Consequently, $\f B$ belongs to $\sk S(\sk K) $, by
\refEq{var2}.  This shows that every simple, subdirect factor of $\f
A$ is in $\sk S(\sk K) $.  Since $\sk S(\sk K) $ is closed under
subalgebras and direct products, it follows that $\f A$ is in $\sk
S(\sk K)$. In other words, every model of $\Delta$ is in $\sk S(\sk
K)$.

To establish the  reverse inclusion, consider first an arbitrary full coset relation algebra
$\cra C F$.  Certainly, $\cra C F$ is  in $\sk S(\sk K)$, by
\refEq{var3}, and hence   is a model of $\Theta$, by \refEq{var2}. If $\mc F$ is  simple in the sense that   the quotient isomorphism system index set $\mc E$ coincides with $I\times I$ (the universal relation on the group system index set $I$), then $\cra C F$ is simple in the
algebraic sense of the word that it has exactly two ideals, by Theorem 6.1 in \cite{ag}.  Each
equation  corresponding to a sentence in $\Theta$ is therefore true of $\cra
C F$, so $\cra C F$ is a model of $\Delta$.

Next, consider the
case when $\mc F$ is not simple.  By Decomposition Theorem 6.2 in \cite{ag}, the algebra
$\cra C F$ is isomorphic to a direct product of coset relation
algebras $\cra C {F\famxi}$, where each $\mc F\famxi$ is
\textit{simple} in the sense that it is a maximal connected component of
$\mc E$. Each algebra $\cra C {F\famxi}$ must
be a model of $\Delta$, by the observations of the preceding paragraph.  Since equations are preserved under the
passage to direct products, it follows that $\cra C F$ is a model
of $\Delta$.  In other words, every full coset relation algebra is a model of
$\Delta$.

Finally, equations are also preserved under that
passage to subalgebras, so any coset relation algebra---that is to say, any algebra embeddable into a full coset
relation algebra---will be a model of $\Delta$.  This
proves \refEq{var4}.  Combine \refEq{var3} and \refEq{var4} to arrive at the desired conclusion of the theorem.
\end{proof}

\newcommand\elp{\ell^+}
\section{\craa\  is not finitely axiomatizable}

We shall need the notion of the Lyndon algebra $\f B$ of a (projective) line $\ell$ (of order at least three) with at least two points.  Let $\ell $ be any finite set, that is to say, any finite projective line, with at least two elements, and take $\ident $ to be a new element not occurring in $\ell$.  The Boolean part of $\f B$ is the Boolean algebra of all subsets of the set $\elp=\ell\cup\{\ident\}$. Singletons $\{p\}$ are identified with the points $p$ themselves. The identity element is taken to be $\ident$, and converse is defined to be the identity function on the universe.  Define the relative product of any two points $p$ and $q$ in $\elp$ as follows:

\begin{equation*}
  p;q=\begin{cases}
   \ell\sim\{p,q\} &\quad \text{if\quad $p\neq q$}\,, \\
   p+\ident &\quad \text{if\quad $p=q$}\,, \\
   p &\quad \text{if\quad $q=\ident$}\,\\
     q &\quad \text{if\quad $p=\ident$}\,.
 \end{cases}
\end{equation*}
Extend $\,;\,$ to a binary operation on the universe by making it
distributive over arbitrary unions.  The resulting algebra $\f B$ is
well known to be a simple relation algebra (see Lyndon\,\cite{l61}).

Fix a Lyndon algebra $\f B$ on a finite line $\ell$ with  at least
two points. Assume, from now till \refT{fa8}, that $\f B$ is
embeddable via a mapping $\vth$ into a full coset relation algebra
$\cra CF$. Since $\f B$ is simple, it may be assumed that the triple
$\mc F$ is simple in the sense that its quotient isomorphism index
set $\mc E$ coincides with the universal relation $I\times I$ on the
group index set $I$. In more detail, if $\vth(1)$ includes an atom
of the form $\r\xy\lph$ for some pair $ \pair xy$ in $ \mc E$, take
$J$ to be the equivalence class of $x$ in $\mc E$, and let $\mc F'$
be the restriction of $\mc F$ to $J$:
\[\mc F'=(G'\smbcomma \vp'\smbcomma C')\comma\] where $G'$ is the
system of groups $\G x$ with $x$ in $J$, and $\vp'$ is the system
of quotient isomorphisms $\vphi\xy$ with $x$, $y$ in $J$, and
similarly for the coset system $C'$. The projection $\pi$ of $\cra
CF$ to $\cra C{F'}$ is a non-trivial homomorphism, since it maps
the atom $\r\xy\lph$ to itself, so the composition $\pi\scir\vth$
is a non-trivial homomorphism, and therefore an embedding, of the
simple algebra $\f B$ into $\cra C{F'}$.

The strategy of the proof is to show that all the subgroups $\h\xy$
are trivial, and $\cra C F$ is representable. Hence $\f B$ has to be
representable since it is embeddable into $\cra C F$. Thus no
non-representable Lyndon algebra can be in $\craa$. We then adapt
Monk's proof in \cite{mo64} that \rraa\ is not finitely
axiomatizable to show that the same applies to \craa. 

\begin{lm}\labelp{L:fa2} If $\h\xy\neq\{e_x\}$\comma then there is a
unique point $p$ in $\ell$ such that $(\G x\times \G y)\cap
\vth(p)\neq \varnot$\per For this point $p$\comma we have\[(\G
x\times\G y)\cup(\G y\times \G x)\seq\vth(p)\] and \[(\G x\cup\G
y)\times(\G x\cup \G y)\seq\vth(p+\ident)\per\]
\end{lm}
\begin{proof} Observe first that the hypothesis on $\h\xy$ implies
that $x\neq y$, since $\h\xx=\{e_x\}$.  The set $U=\tbigcup_{x\in
I}\G x$ is the base set of $\cra CF$. The unit $1$ of $\f B$ is the
sum of the singletons, so it is the set $\elp$, that is to say, it
is the line $\ell$ with the identity element $\ident$ adjoined. Use
this observation, use that $\ell$ is finite and the fact that $\vth$
is an embedding of $\f B$ into $\cra CF$ to obtain
\begin{multline*}\tag{1}\labelp{Eq:fa2.1}
\tbigcup\{\G u\times\G v:u,v\in I\}=(\tbigcup_{u\in I}\G u)\times
(\tbigcup_{v\in I}\G v)=U\times U\\=\vth
(1)=\vth(\{\ident\}\cup\tsum\{p:p\in\ell\})=\{\vth(\ident)\}\cup\tbigcup\{\vth(p):p\in
\ell\}\per
\end{multline*} It is clear  from \refEq{fa2.1} that
\begin{equation*}\tag{2}\labelp{Eq:fa2.2}
(\G x\times\G y)\cap\vth(p)\neq\varnot
\end{equation*}
\[\] for some $p$ in $\elp$.  It is equally clear that $p\neq
\ident$, since $\G x\times\G y$ is disjoint from the identity
relation $\id U$, because the groups $\G x$ and $\G y$ are assumed
to be disjoint, and $\id U$ is the image of $\ident $ under $\vth$.
The set $\G x\times\G y$ is the union of the relations $\r\xy\lph$
for various $\lph$, so there must be an index $\lph$ for which
\[\r\wi\lph\cap\vth(p)\neq\varnot\comma \] by \refEq{fa2.2}. The relation $\r\wi\lph$ is an atom in $\cra CF$, and
the image $\vth(p)$ is an element in $\cra CF$, so
\begin{equation*}\tag{3}\labelp{Eq:fa2.3}
\r\wi\lph\seq\vth(p)\comma\end{equation*} by the definition of an
atom.

Form the converse of both sides of \refEq{fa2.3}, and use monotony,
the embedding properties of $\vth$, and the fact that converse is
the identity function in $\f B$  to obtain
\begin{equation*}\tag{4}\labelp{Eq:fa2.4}
\r\xy\lph\mo\seq\vth(p)\mo=\vth(p\ssm)=\vth(p)\per
\end{equation*}
Apply \refL{cra1}, and then use \refEq{fa2.4}, monotony, the
embedding properties of $\vth$, and the definition of relative
multiplication in $\f B$ to arrive at
\begin{multline*}\tag{5}\labelp{Eq:fa2.5}
\tbigcup\{\r\xx g:g\in \h\xy\}= 
\r\xy\lph\otimes \r\xy\lph\mo
\seq\vth(p)\otimes\vth(p)=\vth(p;p)=\vth(p+\ident)\per
 \end{multline*}
 Use \refEq{fa2.5} and the fact that $\r\xx g$ is disjoint from the identity relation $\id U=\vth(\ident)$
 when $g\neq e_x, g\in\h\xy$, to conclude that
 \begin{equation*}\tag{6}\labelp{Eq:fa2.6}
  \r\xx g\seq\vth(p)
\end{equation*} for $g\neq e_x, g\in\h\xy$.

Assume now for a contradiction that $\r\xy\gm$ is not included in
$\vth (p)$ for some $\gm$. The  first part of the proof shows that
there must be a point $q$ different from $p$ such that
\[\r\xy\gm\seq\vth(q).\] The  argument of the preceding paragraphs, with $q$ in place of $p$, shows that
\begin{equation*}\tag{7}\labelp{Eq:fa2.7}
\r\xx g\seq\vth(q)
\end{equation*}for all $g\neq e_x, g\in\h\xy$.  Choose such a $g$, which certainly exists by
the assumption that $\h\xy\neq\{e_x\}$. We then have
\[\r\xx g\seq\vth(p)\cap\vth (q)=\vth(p\cdot q)=\vth (0)=\varnot,\] by \refEq{fa2.6}, \refEq{fa2.7}, and the embedding properties of $\vth$.  The desired contradiction has arrived, because the relation $\r\xx g$ is not empty.  Conclusion:
\begin{equation*}
\r\xy \gm\seq\vth(p)
\end{equation*} for all  $\gm$, that is to say,
      \begin{equation*}\tag{8}\labelp{Eq:fa2.8}
                  \G x\times\G y\seq\vth(p)\comma
     \end{equation*} by Partition \refL{i-vi}.

     There cannot be another point 
     \[(\G x\times\G y)\cap\vth(q)\neq \varnot,\] for the preceding argument with $q$ in place of $p$ would give
       \begin{equation*}\tag{9}\labelp{Eq:fa2.9}
                  \G x\times\G y\seq\vth(q)\comma
     \end{equation*} and therefore
     \[\G x\times\G y\seq\vth(p)\cap\vth(q)=\vth(p\cdot q)=\vth(0)=\varnot\comma\] by \refEq{fa2.8}, \refEq{fa2.9}, and the embedding  properties of $\vth$.  This is a clear absurdity.

     Finally,
      \begin{equation*}\tag{10}\labelp{Eq:fa2.10}
                \G y\times\G x=(\G x\times\G y)\mo\seq\vth(p)\mo=\vth(p\ssm)=\vth (p),
     \end{equation*}
       and
     \begin{multline*}
      \G x\times\G x=(\G x\times\G y)\rp(\G y\times\G x)=
     (\G x\times\G y)\otimes(\G y\times\G x)\\\seq\vth(p)\otimes\vth(p)=\vth(p;p)=\vth(p+\ident),
     \end{multline*} by \refL{cra2}, \refEq{fa2.8}, \refEq{fa2.10}, the embedding properties of $\vth$, and the definition of relative multiplication in $\f B$. Interchange $x$ and $y$ in this last computation to arrive at $\G y\times\G y\seq\vth(p+\ident)$.  This completes the proof of  the lemma.\end{proof}

     \newcommand\sip{\sim_p}
       \newcommand\nsip{\nsim_p}

     \begin{df}\labelp{D:fa3}For a given point $p$ in $\ell$, define a binary relation $\sip$ on $I$ by $x\sip y$ if and only if $\G x\times\G y\seq\vth(p+\ident)$.\end{df}
     \begin{lm}\labelp{fa4} $\sip$ is an equivalence relation on its domain.\end{lm}\begin{proof}If $x$ is in the domain of $\sip$, then $x\sip y$ for some $y$, and consequently $\G x\times\G y$ is included in $\vth(p+\ident)$, by \refD{fa3}.  Apply \refL{fa2} to see that $\G x\times\G x$ and $\G y\times \G x$ are both included in $\vth(p+\ident)$, so that $x\sip x$ and $y\sip x$.  Thus, $\sip$ is reflexive on its domain, and also symmetric. If $x\sip y$ and $y\sip z$, then both $\G x\times \G y$ and $\G y\times\G z$ are included in $\vth(p)$.  It follows from \refL{cra2}, the preceding inclusions, monotony, the embedding properties of $\vth$, and the definition of relative multiplication in $\f B$ that
     \[\G x\times\G z=(\G x\times\G y)\otimes (\G y\times\G z)\seq \vth(p)\otimes\vth(p)=\vth(p;p)=\vth(p+\ident),\] so that $x\sip z$.  Thus, $\sip $ is transitive. \end{proof}

     \begin{lm}\labelp{L:fa5} For every $x$ in $I$\comma there is a $y$ in $I$ such that $x\nsip y$\per\end{lm}
     \begin{proof} Suppose, for a contradiction, that $x\sip y$ for all $y$ in $I$. In particular, for each $z$ in $I$, we have $x\sip z$.  Use symmetry and transitivity to obtain $y\sip z$ for all $y$ and $z$ in $I$. This means that $\G y\times\G z$ is included in $\vth(p+\ident)$ for all $y$ and $z$  in $I$, by \refD{fa3}, and therefore
     \[U\times U=\tbigcup\{\G y\times\G z:y,z\in I\}\seq\vth(p+\ident)\per\] Thus,
     \[\vth(1)=U\times U= \vth(p+\ident),  \]and therefore $1=p+\ident$, because $\vth$ is an embedding.  But then the line $\ell$ has just one point, namely $p$, in contradiction to the assumption that it has at least two points.
     \end{proof}

     \newcommand\xv{{xv}}
      \newcommand\yv{{yv}}
       \newcommand\vx{{vx}}
        \newcommand\vy{{vy}}

\newcommand\br[1]{\bar{#1}}

     We are close to our goal of proving that all subgroups $\h\xy$ must be trivial.  We need one more lemma.
\begin{lm}\labelp{L:fa6} If $x\sip y$ and $\h\xy\neq\{e_x\}$, then
\[\h\xv=\{e_x\}\comma\quad\h\yv=\{e_y\}\comma\qquad\text{and}\qquad\h\vx=\h\vy=\{e_v\}\] for all $v$ in $I$ such that $x\nsip v$\per\end{lm}
     \begin{proof} Consider an element $v$ in $I$ such that $x\nsip v$, and suppose that   $\h\xv\neq\{e_x\}$ or $\h\vx \neq\{e_v\}$.  Apply \refL{fa2} to obtain a unique $q$ such that
     \begin{align*}
        (\G x\cup\G v)\times(\G x\cup\G v)&\seq\vth(q+\ident)\comma\tag{1}\labelp{Eq:fa6.1}\\
        \intertext{and therefore $x\sim_q v$. Observe that $q\neq p$, since $x\nsip v$.  The assumption that $x\sip y$ implies that $\G x\times \G y$ is included in $\vth(p+\ident)$, by \refD{fa3}, so}
     (\G x\cup\G y)\times(\G x\cup\G y)&\seq\vth(p+\ident)\comma\tag{2}\labelp{Eq:fa6.2}
     \end{align*}
       by \refL{fa2}.  In particular, combine \refEq{fa6.1} and \refEq{fa6.2}, and use the embedding properties of $\vth$, and Boolean algebra, to see that\begin{multline*}
         \G x\times\G x\seq\vth(p+\ident)\cap\vth(q+\ident)=\vth((p+\ident)\cdot (q+\ident))\\=\vth(p\cdot q+p\cdot\ident+q\cdot\ident+\ident\cdot \ident)=\vth(\ident)=\id U\per
           \end{multline*}  This inclusion can   hold only if $\G x$ has just one element, that is to say, it can hold only if $\G x=\{e_x\}$, which would force $\h\xy=\{e_x\}$. The desired contradiction has arrived, because it was assumed that $\h\xy\neq\{e_x\}$, so we must have $\h\xv=\{e_x\}$ and $\h\vx=\{e_v\}$.

           Next, suppose that $\h\yv\neq\{e_y\}$ or $\h\vy\neq\{e_v\}$. We must have $y\nsip v$, by transitivity, since $x\sip y$ and $x\nsip v$.  Apply the preceding argument with $x$ and $y$ interchanged to arrive at a contradiction, and therefore to conclude that $\h\yv=\{e_y\}$ and $\h\vy=\{e_v\}$.
     \end{proof}
     \begin{theorem}\labelp{T:fa7}$\h\xy=\{e_x\}$ for all $x$ and $y$ in $I$\per\end{theorem}\begin{proof}Assume, for a contradiction, that $\h\xy\neq \{e_x\}$, and observe as before that this forces $x\neq y$.  By \refL{fa2}, there is a unique point $p$ such that $\G x\times\G y$ is included in $\vth(p+\ident)$, and consequently $x\sip y$.  There is also a point $v$ such that $x\nsip v$, by \refL{fa5}.  Apply \refL{fa6} to obtain
     \begin{equation*}\tag{1}\labelp{Eq:fa7.1}
       \h\xv=\{e_x\},\qquad\h\yv=\{e_y\},\qquad \h\vx=\h\vy=\{e_v\}\per
\end{equation*} The quotient isomorphism $\vphi\xv$ maps $\G x/\h\xv$ isomorphically to $\G v/\h\vx$ (recall that $\k\xv=\h\vx$, by \refCo{con11}), so it maps $\G x/\{e_x\}$ isomorphically to $\G v/\{e_v\}$, by \refEq{fa7.1}, that is to say, it maps distinct cosets of $\{e_x\}$ to distinct cosets of $\{e_v\}$. Image \refT{domainofmap}, together with \refCo{con11} and \refEq{fa7.1}, implies that
 \begin{equation*}\tag{2}\labelp{Eq:fa7.2}
  \vphi\xv[\h\xv\scir\h\xy]=\k\xv\scir\h \vy=\h\vx\scir\h\vy=\{e_v\}\per
\end{equation*}

\comment{\begin{figure}[htb]  \psfrag*{x}[B][B]{$x$}
\psfrag*{y}[B][B]{$y$} \psfrag*{z}[B][B]{$z$}
\psfrag*{a}[B][B]{$\{e_x\}$} \psfrag*{b}[B][B]{$\{e_v\}$}
\psfrag*{c}[B][B]{$\{e_v\}$} \psfrag*{d}[B][B]{$\{e_y\}$}
\psfrag*{e}[B][B]{$\h\xy$} \psfrag*{f}[B][B]{$\h\yx$}
\begin{center}\includegraphics*[scale=.8]{triangle}\end{center}
\caption{The triangle of the Image Theorem.}\labelp{F:fig1}
\end{figure}}

The composite subgroup $\h\xv\scir\h\xy$ is a union of cosets of $\h\xv$, and $\vphi\xv$ maps
distinct cosets of $\h\xv$ to distinct cosets of $\h\vx$, so \refEq{fa7.2} and the isomorphism
properties of $\vphi\xv$ imply that $\h\xv\scir\h\xy$ must be a coset of $\h\xv$, and in fact
it must be the identity coset $\{e_x\}$.  Thus, $\h\xy=\{e_x\}$, in contradiction to the assumption
that these two subgroups are distinct.
\end{proof}

A relation algebra is called \emph{completely representable} if it
has a representation in which all existing suprema are taken to set
theoretic unions.

  \begin{theorem}\labelp{T:fa8}If a Lyndon algebra $\f B$ of a finite line with at least two points is embeddable into a
  full coset relation algebra $\cra CF$\comma then $\cra CF$ is completely representable and in fact it is isomorphic to a
  full group relation algebra\per Hence\comma $\f B$ is representable\per\end{theorem}
  \begin{proof} Because $\f B$ is simple, it may be assumed that the group triple $\mc F$ is simple as well, that is to say, its quotient isomorphism index set is the universal relation on the group index set $I$ (see the remarks at the beginning of the section).  The normal subgroups $\h\xy$ are all trivial, by \refT{fa7}.
  The definition of the atomic relations $\r \xy\lph$ therefore implies that
  \begin{multline*}\r\xy\lph=\tbigcup\{\hs\xy\gm\times(\ks\xy\gm\scir\ks\xy\lph):\gm<\kai\xy\}\\=\tbigcup\{\{g\}\times\{\br g\scir\br f\}:g\in\G x\}
  =\{\pair g{\br g\scir\br f}:g\in \G x\}\comma\end{multline*} where
  $\lph=\{ f\}$ and the quotient isomorphism $\vphi\xy$ maps each element $\{g\}$ in $\G x/\{e_x\}$ to the
  corresponding element $\{\br g\}$ in $\G y/\{e_y\}$.  Such an atom is clearly a function, so $\cra CF$ is an atomic relation algebra with functional atoms, by
  Boolean Reduct \refT{disj}.  The J\'onsson-Tarski\,\cite{jt52} Representation Theorem for atomic relation
  algebras  with functional  atoms, in the form given by Andr\'eka-Givant\,\cite{ag13}, implies that $\cra CF$ is completely
  representable.
An atomic measurable relation algebra is completely representable if and only if it has a scaffold,
  which in turn happens if and only if its completion is isomorphic to a full  group relation algebra, by
  Scaffold Representation Theorem 7.6, Corollary 7.7, and Theorem 7.8 in \cite{ga}.  Thus, $\cra CF$
  (which, being complete, is its own completion) is isomorphic to a full group relation algebra, and
  consequently $\f B$ is representable since it is isomorphic to a subalgebra of $\cra CF$.
  \end{proof}

  \begin{cor}\labelp{C:ly} No finite non-representable Lyndon algebra of a line with at least two points is in
  $\craa$.
  \end{cor}

  The only properties of $\f B$ that are used in the proofs leading up to \refT{fa8} are that the unit $1$ of $\f B$ is
  the sum of finitely many equivalence elements $\cs ei=\cs pi +\ident$ for $1\le i\le n$ and some $n\ge 2$, and these
  equivalence elements satisfy the equation $\cs ei\cdot\cs ej=\ident$ for $i\neq
  j$.
  \begin{cor} Let $\cra CF$ be a full coset relation algebra on a simple group triple $\mc F$\per
  If in $\cra CF$ the unit is the sum of finitely many reflexive equivalence elements for which the
  pairwise distinct meets are  always the identity element\comma then $\cra CF$ is completely representable
  and is isomorphic to a full  group relation algebra\per
  \end{cor}

  \newcommand{\K}{\mbox{\sk K}}
  \newcommand{\elK}{\mbox{\sk E\sk l\K}}


\begin{theorem}\labelp{T:fa10}
\craa\ is not finitely axiomatizable\per Moreover, if $\K$ is any
class such that $\rraa\subseteq\K\subseteq\craa$, then $\K$ is not
finitely axiomatizable.
\end{theorem}

  \begin{proof}The proof is a modified version of Monk's proof that the class \rraa\
  of representable relation algebras is not finitely axiomatizable.
  Assume $\rraa\subseteq\K\subseteq\craa$. 
%
  Let $\langle \f B_n:n\in\mathbb N\rangle $ be
  an infinite sequence of finite non-representable Lyndon algebras of lines with at least $n+2$ points, indexed by the
  set $\mathbb N$ of natural numbers.  Such a sequence exists by the Bruck-Ryser Theorem (for more details, see
  Monk\,\cite{mo64}). None of the algebras in this sequence can belong to \craa, by 
  \refC{ly}.
  %
  Let $D$ be a non-principal ultrafilter in the Boolean algebra of subsets of $\mathbb N$, and form the ultraproduct
  \[\f A=(\tprod_{n\in \mathbb N}\f  B_n)/D\per\]  Monk\,\cite{mo64} proved that $\f A$ is representable.  Consequently, $\f A$
  belongs to \rraa, which 
  is a subclass of $\K$ by our assumption.
%
  Hence, the complement of $K$ is not closed under ultraproducts, and so $\K$ cannot be finitely
  axiomatized by a well-known theorem of model theory (again, see Monk\,\cite{mo64} for details). 
  Since \rraa\ coincides with \graa\ which is a subclass of \craa,
  we have $\rraa\subseteq\craa\subseteq\craa$, hence \craa\ is not
  finitely axiomatizable.
  \end{proof}

  We can also use \refC{ly} to prove an analogue of J\'onsson's
  theorem \cite[Theorem 3.5.6]{J}.

\begin{theorem}\labelp{T:fa11}
Any equational axiom system for \craa\ must use infinitely many
variables\per Moreover, if $\K$ is any class such that
$\rraa\subseteq\K\subseteq\craa$, then $\K$ is not axiomatizable by
any set of universal formulas that contains only finitely many
variables.
\end{theorem}

\begin{proof}The proof is a modified version of J\'onsson's proof that the class \rraa\
  of representable relation algebras is not axiomatizable by any set of equations containing finitely many variables.
  Assume $\rraa\subseteq\K\subseteq\craa$. In the proof of
  \cite[Theorem 3.5.6]{J}, J\'onsson shows that for any natural
  number $k$ there is a finite non-representable Lyndon algebra $\f B$ of a
  finite line with more than 2 points such that each $k$-generated
  subalgebra of $\f B$ is representable. By \refC{ly}, this algebra
  $\f B$ is not in $\K$, but all $k$-generated subalgebra of $\f B$
  does belong to $\K$, by our assumption
  $\rraa\subseteq\K\subseteq\craa$. This proves that $\K$ cannot be
  axiomatized by any set $\Sigma$ of universal formulas such that $\Sigma$
  contains at most $k$ variables. Since $k$ can be chosen to be any natural number,
  we get that $\K$ cannot be axiomatized with any set of universal formulas that
  contains only finitely many variables.  Since equations are universal
  formulas and $\rraa\subseteq\craa$, we get that $\craa$ cannot be
  axiomatized with any set of equations that contains only finitely
  many variables.
  \end{proof}

It is shown in \cite{AGN} that there are as many varieties between
$\rraa$ and $\craa$ as possible, i.e., continuum many. By our
theorems above, none of these continuum many varieties can be
axiomatized by a set of equations containing finitely many variables
only, in particular, none of them is finitely axiomatizable.

We use infinitely many non-representable coset relation algebras
when constructing the above continuum many varieties. However, any
ultraproduct of these is also non-representable, because the
``cause" of the non-representability in these algebras is
expressible by a common first-order formula. This leaves open the
following.

  \begin{prb} \label{pr} Is $\rraa$ finitely axiomatizable over $\craa$, i.e.,
  is there a finite set $\Sigma$ of equations such that $\rraa$ is exactly those
  coset relation agebras that satisfy $\Sigma$?
  \end{prb}

  In the proof of the present \refT{var}, we also prove that $\craa$
  is the variety generated by the atomic measurable relation
  algebras. Problem 8.5 in \cite{ag} asks whether each measurable
  relation algebra can be embedded into an atomic measurable
  relation algebra. In the light of \refT{var}, this problem is equivalent with asking
  whether there is an equation that holds in all atomic measurable relation algebras but
  not in all measurable relation algebras.

  \begin{prb}\label{pr1} Is \craa\ the variety generated by the
  class of measurable relation algebras?
  \end{prb}

  \bigskip


\end{document}